\documentclass[12pt]{amsart}
\usepackage[dvips]{graphicx}
\usepackage{amsmath,graphics}
\usepackage{amsfonts,amssymb}
\usepackage{xypic}
\usepackage{comment}
\specialcomment{proofc}{}{}
\includecomment{proofc}
\theoremstyle{plain}
\newtheorem*{theorem*}{Theorem}
\newtheorem*{lemma*} {Lemma}
\newtheorem*{corollary*} {Corollary}
\newtheorem*{proposition*}{Proposition}
\newtheorem*{conjecture*}{Conjecture}
\newtheorem{theorem}{Theorem}[section]
\newtheorem{lemma}[theorem]{Lemma}
\newtheorem*{theorem1*}{Theorem 1}
\newtheorem*{theorem2*}{Theorem 2}
\newtheorem*{theorem3*}{Theorem 3}
\newtheorem{corollary}[theorem]{Corollary}

\newtheorem{question}[theorem]{Question}

\theoremstyle{remark}

\newtheorem{example*}{Example}

\theoremstyle{definition}

\def\co{\colon}
\def\gl{\mbox{GL}} \def\Q{\Bbb{Q}}  \def\Z{\Bbb{Z}}  
  \def\l{\lambda} \def\ll{\langle} \def\rr{\rangle}
 \def\a{\alpha} \def\g{\gamma}  \def\bp{\begin{pmatrix}}
\def\sm{\setminus} \def\ep{\end{pmatrix}} \def\bn{\begin{enumerate}} 
   \def\en{\end{enumerate}}
\def\ba{\begin{array}} \def\ea{\end{array}} 
   \def\a{\alpha} \def\b{\beta} \def\ti{\tilde}
\def\id{\mbox{id}}   
  
\def\be{\begin{equation}} \def\ee{\end{equation}} 
  \def\ord{\mbox{ord}} 
   
 \def\aut{\mbox{Aut}}  
\def\Tor{\mbox{Tor}}

\def\tpm {[t^{\pm 1}]}

\def\qkt{\Q\tpm^k}

\def\qt{\Q\tpm}

\def\op{\operatorname}

\begin{document}
\title{Twisted Alexander invariants detect trivial links}
\author{Stefan Friedl}
\address{Mathematisches Institut\\ Universit\"at zu K\"oln\\   Germany}
\email{sfriedl@gmail.com}

\author{Stefano Vidussi}
\address{Department of Mathematics, University of California,
Riverside, CA 92521, USA} \email{svidussi@math.ucr.edu} \thanks{S. Vidussi was partially supported by NSF grant
DMS-0906281.}
\date{\today}


\begin{abstract}
It follows from earlier work of Silver--Williams and the authors that twisted Alexander polynomials detect the unknot and the Hopf link.
We now show that twisted Alexander polynomials also detect the trefoil and the figure-8 knot,
that twisted Alexander polynomials detect whether a link is split and that twisted Alexander modules detect trivial links.
\end{abstract}

\maketitle

%===================================================
\section{Introduction and main results}

An \emph{(oriented) $m$-component link} $L=L_1\cup \dots \cup L_m\subset S^3$ is a collection of $m$ disjoint smooth oriented closed circles in $S^3$. Given
 such link $L$ we denote by $\phi_L$ the canonical epimorphism $\pi_1(S^3\sm L) \to\langle t \rangle$ which is given by sending each meridian to $t$. Given a representation $\a\co \pi_1(S^3\sm L)\to \gl(k,\Q)$ we will introduce in
 Section  \ref{section:deftwialex} the corresponding twisted Alexander $\qt$--module $H_1^{\a \otimes \phi_L}(S^3 \sm L;\qkt)$.

The purpose of this paper is to discuss to what degree the collection of twisted Alexander  modules detects
various types of links. The model example is the following: We can extract information from these modules by looking at their order; in particular, we can define the one--variable twisted Alexander polynomial $\Delta_L^\a(t) \in \qt$. Silver and Williams \cite{SW06} proved that the collection of twisted Alexander polynomials detects the trivial knot among $1$--component links, i.e. knots.
More precisely,  if $L\subset S^3$ is a knot then
$L$ is the unknot  if and only if  $\Delta_L^\a(t) = 1$
for all representations $\a\co \pi_1(S^3\sm L)\to \gl(k,\Q)$.

We thus see that twisted Alexander polynomials detect the unknot, and in a similar vein we showed in \cite{FV07}
 twisted Alexander polynomials detect the Hopf link.
It is natural to ask whether twisted Alexander modules characterize other classes of knots and links. The purpose of this paper is to discuss a number of cases where the answer is affirmative. We will present now the main results, referring to the following sections for the precise statements. The first result, that significantly improves upon \cite[Theorem 1.3]{FV07} is the following:

\begin{theorem}\label{thma}
Twisted Alexander polynomials detect the trefoil and the figure-8 knot.
\end{theorem}

The second result asserts that twisted Alexander modules detect split links  (recall that a link $L$ is \emph{split} if there exists a 2-sphere $S\subset S^3$ such that each component of $S^3\sm S$ contains at least one component of $L$).  Denote by $\op{rk}(L,\a)$  the rank of the twisted Alexander module, i.e. \[ \op{rk}(L,\a) := \mbox{rk}_{\qt} H_1^{\a \otimes \phi_L}(S^3\sm L;\qkt). \] We have the following:

\begin{theorem}\label{thmb}
A link $L$ is a split link if and only if  for any representation $\a\co \pi_1(S^3\sm L)\to \gl(k,\Q)$ we have $\op{rk}(L,\a) > 0$.
\end{theorem}

(A more detailed result, relating $\op{rk}(L,\a)$ with the splittability of $L$, is presented in Section \ref{section:deftwialex}.)

Note that, as the condition $\op{rk}(L,\a) > 0$ is equivalent to the vanishing of  $\Delta_L^\a(t)$, this result simultaneously asserts that twisted Alexander polynomial cannot distinguish inequivalent split links, in particular they fail to characterize the trivial link with more than one component. However, whenever the twisted Alexander module is not torsion, we can define a sort of secondary invariant, defined as the order  of the torsion part of the twisted Alexander module. More precisely we consider the following invariant:
\[  \ti{\Delta}_L^\a(t) := \mbox{ord}_{\qt}\left( \mbox{Tor}_{\qt} H_1^{\a \otimes \phi_L}(S^3\sm L;\qkt)\right). \]
(We refer to Section \ref{section:deftwialex} for details.) We can now formulate our  main theorem.

\begin{theorem}\label{thmc}
An $m$-component link $L$ is trivial if and only if
for any representation $\a\co \pi_1(S^3\sm L)\to \gl(k,\Q)$
we have $\op{rk}(L,\a)=k(m-1)$ and $\ti{\Delta}_L^\a(t)= 1$.
\end{theorem}

In order to prove the theorems above we will build on the results of \cite{FV12a,FV12b}, where we showed that twisted Alexander polynomials determine the Thurston norm and detect the existence of fibrations for irreducible $3$--manifolds with non--empty toroidal boundary. These results  in turn rely on the virtual fibering theorem of Agol \cite{Ag08}
and the work of Wise and Przytycki-Wise \cite{Wi09,Wi12a,Wi12b,PW12}.

We conclude this introduction with some observations tying in the results above with some group--theoretic aspects. First, the fact that twisted Alexander polynomials detect the unknot and the Hopf link is perhaps not entirely surprising, as these are the only links whose fundamental group is abelian. Instead, the fundamental group of all non--trivial knots is non--abelian, hence detection of the trefoil and the figure-8 knot appears far more challenging. Similarly, the unlink is characterized by the fact that $\pi_1(S^3\sm L)$ is a free group, but in  general it is difficult to distinguish a free group from other non-abelian groups.
(We refer to \cite{AFW12} and references therein for a survey on 3-manifold groups, from which these observations can be easily deduced.)

\subsection*{Convention.} Unless specified otherwise, all spaces are assumed to be compact and connected, and links are assumed to be oriented.
%Furthermore we allow norms to be degenerate, i.e. we refer to seminorms as norms.
%By a free abelian group we always mean a non--trivial free abelian group.

\subsection*{Acknowledgment.}
The first author wishes to thank IISER Pune for its generous hospitality.

%=====================================
\section{Preliminaries}

%=====================================

%=====================================
\subsection{The definition of twisted Alexander modules and polynomials} \label{section:deftwialex}

In this section we quickly recall the definition of the twisted Alexander modules and polynomials for links, referring to \cite{Tu01,Hi02,FV10} for history, details and generalizations.

Let $L \subset S^3$ be an oriented $m$--component link.
Consider the canonical morphism $\phi_{L}\co \pi_1(S^3 \sm L)\to \Z = \langle t \rangle $ sending the meridian of each component to $t$ and let $\a\colon \pi_1(S^3 \sm L)\to \gl(k,\Q)$ be a representation.
 Using the tensor representation
\[ \ba{rcl} \a\otimes \phi_{L}\colon  \pi_1(S^3 \sm L) &\to & \gl(k,\qt) \\
g&\mapsto & \a(g)\cdot \phi_{L}(g)\ea \]
we can define the  homology groups  $H_{\ast}^{\a\otimes \phi_L}(S^3 \sm L;\qkt)$  of $S^3 \sm L$ with coefficients in $\qkt$, which inherit from the system of coefficients an action of $\qt$ and, as $\qt$ is Noetherian, are finitely presented as $\qt$--modules.
We refer to these modules as  \emph{twisted Alexander modules of $(L,\a)$}.

We now define
\[ \ba{rcl}
{\Delta}_{L,i}^\a&:=&\mbox{ord}_{\qt} H_i^{\a\otimes \phi_L}(S^3 \sm L;\qkt),\\[2mm]
\ti{\Delta}_{L,i}^\a&:=&\mbox{ord}_{\qt} \mbox{Tor}_{\qt} H_i^{\a\otimes \phi_L}(S^3 \sm L;\qkt), \\[2mm]
\op{rk}(L,\a,i)&:=& \mbox{rk}_{\qt} H_i^{\a\otimes \phi_L}(S^3 \sm L;\qkt).\ea \]
(For the definition of the order $\mbox{ord}_{\qt}(H)$ of a $\qt$-module $H$ we refer to \cite{Tu01}.)
We refer to $\Delta_{L,i}^\a$ as the \emph{$i$--th twisted Alexander polynomial of $(L,\a)$}.
 Note that  $\Delta_{L,i}^\a\in \qt$ and  $\ti{\Delta}_{L,i}^\a\in \qt$ are well-defined up to multiplication by a unit in $\qt$.

(Throughout this paper we drop the $i$ from the notation when $i=1$,
and drop $\a$ from the notation if $\a$ is the trivial one-dimensional representation over $\Q$.)

We conclude this section with an elementary observation.
Let  $\a\colon \pi_1(S^3 \sm L)\to \gl(k,\Q)$ and $\b\colon \pi_1(S^3 \sm L)\to \gl(l,\Q)$ be two representations.
We can then  also consider the diagonal sum representation
$\a\oplus \b\colon \pi_1(S^3 \sm L)\to \gl(k+l,\Q)$. It follows immediately from the definitions that
\be \label{equ:abplus}
{\Delta}_{L,i}^{\a\oplus \b}={\Delta}_{L,i}^\a\cdot {\Delta}_{L,i}^\b.\ee

%=====================================
\subsection{Degrees of twisted Alexander polynomials and the 0-th twisted Alexander polynomial}
We will make use of the following lemma.

\begin{lemma}\label{lem:h0}
Let $L \subset S^3$ be a link and   let  $\a\colon \pi_1(S^3 \sm L)\to \gl(k,\Q)$ be a representation, then
$H_0^{\a\otimes \phi_L}(S^3 \sm L;\qkt)$ is $\qt$-torsion and
\[ \deg (\Delta_{L,0}^\a) \leq k.\]
\end{lemma}

\begin{proof}
Recall that if
 $X$ is a  space and  $\g:\pi_1(X)\to \aut(V)$  a representation,
 then it is well-known (see e.g.  \cite[Section~VI]{HS97}) that
\be \label{equ:hs} H_0^\g(X;V)=V/\{(\g(g)-\id_k)v\,|\, g\in \pi_1(X)\mbox{ and }v\in V\}.\ee
In particular in our case, we pick  $g\in \pi_1(S^3 \sm L)$ such that $\phi_L(g)=t$. It then follows from
(\ref{equ:hs}) and the definition of the Alexander polynomial that
\[ \Delta_{L,0}^\a\,|\, \det((\a\otimes \phi_L)(g)-\id_k).\]
Note that $(\a\otimes \phi_L)(g)=\a(g)t$, in particular $\det((\a\otimes \phi_L)(g)-\id_k)$ is a polynomial of degree $k$.
It now follows that  $\Delta_{L,0}^\a\ne 0$ and that
\[ \deg \Delta_{L,0}^\a\leq \deg \left(\a(g)t-\id_k)\right)=k.\] \end{proof}

%=====================================
\subsection{The Thurston norm, fibered classes and twisted Alexander polynomials}

Let $L \subset S^3$ be an oriented $m$--component link. Recall that the link $L$ is \textit{fibered} if its complement can be fibered over $S^1$ by Seifert surfaces of the link. (Note that, when $m \geq 2$, this is stronger than the requirement that $S^3 \sm L$ admits a fibration: precisely, it is equivalent to requiring  that the class of $H^1(S^3 \sm L;\Z)$ determined by the canonical morphism  $\phi_{L} \co  \pi_1(S^3\sm L) \to\langle t \rangle$  is fibered.)

The following theorem is a consequence of Theorems~1.1 and 1.2, Proposition 2.5 and Lemma 2.8 of  \cite{FK06} (see also \cite{Fr12} for an alternative proof).

\begin{theorem}\label{thm:fk06}
Let $L \subset S^3$ be an oriented $m$--component link and $\a\colon \pi_1(S^3 \sm L)\to \gl(k,\Q)$ a representation
such that $\Delta_{L}^\a\ne 0$. Then
\be \label{equ:x} \deg \Delta_{L}^\a-  \deg \Delta_{L,0}^\a\leq k\|\phi_{L}\|_T.\ee
Furthermore, if $L$ is a fibered link, then $\Delta_{L}^\a\ne 0$ and (\ref{equ:x}) is an equality.
\end{theorem}

Here, $\|\phi_{L}\|_T$ is the Thurston norm of the class $\phi_{L}$ (see \cite{Th86}): this norm is related with the genus of the link $g(L)$ by the equation $\|\phi_L\|_T = 2g(L) - 2 + m$.

The above theorem thus says that degrees of twisted Alexander polynomials give lower bounds on the genus of the link,
and that they determine it for fibered links.
Using  work of Agol \cite{Ag08}, Liu \cite{Liu11}, Przytycki-Wise \cite{PW11,PW12} and Wise \cite{Wi09,Wi12a,Wi12b}
the authors proved in \cite{FV12a,FV12b} in particular that twisted Alexander polynomials decide the fiberability and  determine the genus of a link. Specifically we have the following:

\begin{theorem} \label{thm:detectx} Let $L \subset S^3$ be an oriented $m$--component link.
Then there exists a  representation $\a\colon \pi_1(S^3 \sm L)\to \gl(k,\Q)$ such that $\Delta_L^\a\ne 0$ and such that
\[  \deg \Delta_{L}^\a-  \deg \Delta_{L,0}^\a=k\|\phi_L\|_T.\]
Furthermore, if $L$ is not fibered, there exists a  representation $\a' \colon \pi_1(S^3 \sm L)\to \gl(k,\Q)$ such that
\[  \Delta_{L}^{\a'}=0.\]
\end{theorem}

This theorem has the following corollary, whose second part refines the main theorem of \cite{FV07} inasmuch as it asserts the sufficiency of the use of one--variable twisted Alexander polynomials.

\begin{corollary}\label{cor:unknothopf}
\bn
\item Let $K\subset S^3$ be a knot. If $K$ is trivial, then
for any representation $\a\co \pi_1(S^3\sm  K)\to \gl(k,\Q)$ we have
$\Delta_K^\a = 1$. Conversely, if $K$ is non-trivial, then
there exists a representation $\a\co \pi_1(S^3\sm  K)\to \gl(k,\Q)$
such that $\Delta_{K}^\a $ is not a constant.
\item Let $L\subset S^3$ be a 2-component link. If $L$ is the Hopf link, then
for any representation $\a\co \pi_1(S^3\sm L)\to \gl(k,\Q)$ we have
\[ \tau_{L}^{\a} := \Delta_L^\a (\Delta_{L,0}^{\a})^{-1} = 1. \]
Conversely, if $L$ is not the Hopf link, then
there exists a representation $\a\co \pi_1(S^3\sm L)\to \gl(k,\Q)$
such that $\tau_{L}^{\a}\neq  1$.
\en
\end{corollary}

The reader may have noticed that the invariant $\tau_{L}^{\a}$ introduced in the statement of the corollary is, in fact, the twisted Reidemeister torsion; see e.g. \cite{FV10} for a discussion of this point of view.

\begin{proof}
Let $K\subset S^3$ be a knot. If $K$ is trivial, then all first twisted homology modules are zero, hence all twisted Alexander polynomials are equal to $1$.
Conversely, if $K$ is non-trivial, then the genus is greater than zero, and it then follows immediately from Theorem \ref{thm:detectx} that there exists a rational representation
with corresponding non-constant twisted Alexander polynomial.

Now let $L\subset S^3$ be a 2-component link. Then it is well-known that the following are equivalent:
\bn
\item[(a)] $L$ is the Hopf link,
\item[(b)] $S^3\sm L\cong T^2\times I$,
\item[(c)] $L$ is fibered with $||\phi_L||_T= 2 g(L) = 0$.
\en
It follows easily from the implication (a) $\Rightarrow$ (b) that the twisted Alexander modules of the Hopf link are the homology groups of the infinite cyclic cover $T^2$ determined by $\phi_{L}$, i.e. homotopically a copy of $S^1$.  Given any representation $\a\co \pi_1(S^3\sm L)\to \gl(k,\Q)$ it follows
 that $\tau_{L}^{\a} = 1$ (we refer to \cite[p.~644]{KL99} for details).
Now suppose that $L$ is not the Hopf link. Then $\phi_{L}$  is either not  fibered
or $||\phi_{L}||_T> 0$. It follows from
 Theorem \ref{thm:detectx} that there  exists a representation $\a\co \pi_1(S^3\sm L)\to \gl(k,\Q)$
such that $\Delta_{L}^\a$ is either zero  or such that
\[ \deg(\Delta_{L}^\a)-\deg(\Delta_{L,0}^\a)>0.\] Either way, $\tau_{L}^{\a} \neq 1$.
\end{proof}

%=====================================
\section{Proofs of the main results}

%=====================================
\subsection{Twisted Alexander polynomials detect the trefoil and the figure-8 knot} \label{section:trefoilfigure8}
The following theorem is the   promised more precise version of Theorem \ref{thma}.

\begin{theorem}\label{thmatext}
Let $K$ be a knot. Then $K$ is equivalent to the trefoil knot (the figure-8 knot respectively) if and only if the following conditions hold:
\bn
\item $\Delta_K(t)= 1-t+t^2$ ($\Delta_K(t)= 1-3t+t^2$ respectively)
\item for any representation $\a\co \pi_1(S^3\sm  K)\to \gl(k,\Q)$ we have
\[  \Delta_K^\a(t)\ne 0 \mbox{ and }\deg \Delta_K^\a(t)\leq 2k.\]
\en

\end{theorem}

\begin{proof}
Let $K$ be the trefoil knot or the figure-8 knot.
It is well known that in the former case
 $\Delta_K(t)= 1-t+t^2$ and that in the latter case $\Delta_K(t)= 1-3t+t^2$.
Note that in either case $K$ is a fibered genus one knot.
It now follows from Theorem \ref{thm:detectx} that for
any representation $\a\co \pi_1(S^3\sm  K)\to \gl(k,\Q)$ we have
$\Delta_K^\a(t)\ne 0$ and that
\[ \deg \Delta_K^\a(t)-\deg \Delta_{K,0}^\a(t)=k(2\,\mbox{genus}(K)-1)=k.\]
We deduce from Lemma \ref{lem:h0} that $\deg \Delta_{K,0}^\a(t)\leq k$. We thus obtain the desired inequality
\[ \deg \Delta_K^\a(t)\leq 2k.\]
This concludes the proof of the `only if' direction of the theorem.

Now suppose that $K$ is a knot such that  for any representation $\a\co \pi_1(S^3\sm  K)\to \gl(k,\Q)$ we have
\[  \Delta_K^\a(t)\ne 0 \mbox{ and }\deg \Delta_K^\a(t)\leq 2k.\]
It follows from Theorem  \ref{thm:detectx} that $K$ is fibered and that the genus of $K$ equals one.
From \cite[Proposition~5.14]{BZ85} we deduce that $K$ is equivalent to either the trefoil knot or the figure-8 knot.
The `if' direction of the theorem now follows from the fact mentioned above that the ordinary Alexander polynomial distinguishes the trefoil knot from the figure-8 knot.
\end{proof}

%=====================================
\subsection{Split links}\label{section:splitlink}

We say that a link $L$ is \emph{$s$-splittable} if there exist $s$ disjoint $3$-balls $B_1,\dots,B_s\subset S^3$
such that each $B_i$ contains \textit{at least} one component of $L$ and such that $S^3\sm (B_1\cup \dots \cup B_s)$ also contains a component of $L$.
Furthermore we say that $L$ is \emph{$s$-split} if $L$ is $s$-splittable but not $(s+1)$-splittable.

The following theorem implies in particular
Theorem \ref{thmb}.

\begin{theorem}\label{thm:split}
Let $L\subset S^3$ be an oriented $m$--component link. Then the following hold:
\bn
\item If $L$ is $s$-splittable, then for any representation $\a\co \pi_1(S^3\sm L)\to \gl(k,\Q)$ we have
\[ \op{rk}(L,\a)\geq sk.\]
\item If $L$ is $s$-split, then there exists a representation $\a\co \pi_1(S^3\sm L)\to \gl(k,\Q)$ such that
\[ \op{rk}(L,\a)=sk.\]
\en
\end{theorem}

\begin{proof}
Denote as usual by  $\phi_{L}\co \pi_1(S^3\sm L)\to \ll t\rr$ the map which is given by sending each meridian to $t$. By slight abuse of notation,
we will also denote by $\phi_{L}$ the restriction of $\phi_L$ to any subset of $S^3\sm L$.

Suppose that  $L\subset S^3$  is an $s$-splittable link.
We pick disjoint $3$-balls $B_1,\dots,B_s\subset S^3$
such that each $B_i$ contains at least one  component of $L$ and such that $B_0:=S^3\sm (B_1\cup \dots \cup B_s)$ also contains a component of $L$.
For $i=1,\dots,s$ we write $S_i:=\partial B_i$ and for $i=0,\dots,s$ we write $L_i:=L\cap B_i$.
By assumption $L_i$ is non-empty for any $i$.

Now let  $\a\co \pi_1(S^3\sm L)\to \gl(k,\Q)$ be a representation.
 We  consider  the following Mayer-Vietoris sequence
\[ \ba{cclcccc}\hspace{-0.2cm}& \bigoplus\limits_{i=1}^s H_1(S_i;\qkt)\hspace{-0.2cm}&\to \hspace{-0.2cm}&\bigoplus\limits_{i=0}^s H_1(B_i\sm L_i;\qkt) \hspace{-0.2cm}&\to\hspace{-0.2cm}&H_1(S^3\sm L;\qkt)\hspace{-0.2cm}&\to \\
\to\hspace{-0.2cm}&\bigoplus\limits_{i=1}^s H_0(S_i;\qkt)\hspace{-0.2cm}&\to \hspace{-0.2cm}& \bigoplus\limits_{i=0}^s H_0(B_i\sm L_i;\qkt)\hspace{-0.2cm}&\to\hspace{-0.2cm}&\hspace{-0.2cm}\dots& \ea
 \]
where the representation is given by $\a\otimes \phi_L$ in each case.
Note that the restriction of ${\a\otimes \phi_L}$ to $\pi_1(S_i)$, $i=1,\dots,s$ is necessarily trivial,
but that the restriction of $\phi_L$ to $\pi_1(B_i\sm L_i)$, $i=0,\dots,s$ is non-trivial since $L_i$ consists of at least one component.
 It follows immediately from the definition of homology with coefficients that
for $i=1,\dots,s$ we have $H_0(S_i;\qkt)\cong \qkt$ and $H_1(S_i;\qkt)\cong 0$.

Finally note that for $i=0,\dots,s$ and $j=0,1$ we have inclusion induced isomorphisms
\[ H_j(B_i\sm L_i;\qkt)\xrightarrow{\cong} H_j(S^3\sm L_i;\qkt).\]
This entails, by Lemma \ref{lem:h0} that
for $i=0,\dots,s$ the modules
$H_0(B_i\sm L_i;\qkt)$ are torsion $\qt$-modules.
We thus see that the above Mayer-Vietoris sequence gives rise to an exact sequence
\be \label{equ:mv}0\to \bigoplus\limits_{i=0}^s H_1(S^3\sm L_i;\qkt)\to H_1(S^3\sm L;\qkt)\to \qt^{ks}\to T\ee
where $T$ is a torsion $\qt$-module.
In particular we now deduce that
\[ \op{rk}(L,\a)=\op{rk}_{\qt}\left(H_1(S^3\sm L;\qkt)\right)\geq \op{rk}_{\qt}\left(\bigoplus_{i=1}^s H_0(S_i;\qkt)\right)=sk.\]
This concludes the proof of (1).

We now  suppose that $L$ is in fact an $s$-split link.
Note that we have a canonical homeomorphism
\[ S^3\sm L\cong S^3\sm L_0 \,\# \,\dots \,\# \,S^3\sm L_s.\]
Furthermore it is straightforward to see that
the fact that $L$ is not $(s+1)$-splittable implies that  the manifolds $S^3\sm L_i$, $i=0,\dots,s$ are irreducible.

It follows from Theorem \ref{thm:detectx} that for $i=0,\dots,s$ there exists a representation
$\a_i:\pi_1(S^3\sm L_i)\to \gl(k_i,\Q)$ such that $\Delta_{L}^{\a_i}\ne 0$.
We now denote by $k$ the greatest common divisor of the $k_i$. After replacing $\a_i$ by
the diagonal sum of $k/k_i$-copies of the representation $\a_i$ we can in light of (\ref{equ:abplus})  assume that in fact $k=k_i$, $i=0,\dots,s$.
We now denote by
\[ \a\co \pi_1(S^3\sm L)\to \gl(k,\Q)\]
the unique representation which has the property that for $i=0,\dots,s$ the restriction of $\a$
to $\pi_1(B_i\sm L_i)$ agrees with the restriction of $\a_i$ to $\pi_1(B_i\sm L_i)$.
By the above the modules $H_1(S^3\sm L;\qkt)$ are $\qt$-torsion modules. It now follows
from (\ref{equ:mv}) that
\[ \op{rk}(L,\a)=\op{rk}_{\qt}\left(H_1(S^3\sm L;\qkt)\right)=\op{rk}_{\qt}\left(\bigoplus_{i=1}^s H_0(S_i;\qkt)\right)=sk.\]
This concludes the proof of (2).
\end{proof}

%=====================================
\subsection{Detecting unlinks}

We finally turn to the problem of detecting unlinks. The following well-known lemma gives a purely group-theoretic characterization of unlinks.

\begin{lemma}
A link $L$ is trivial if and only if $\pi_1(S^3\sm L)$ is a free group.
\end{lemma}

\begin{proof}
The `only if' direction is obvious. So suppose that  $L=L_1\cup \dots\cup L_m$ is an $m$-component link such that $\pi_1(S^3\sm L)$ is a free group.
We have to show that each $L_i$ bounds a disk in the complement of the other components. We denote by $T_i$ the torus which is the boundary of a tubular neighborhood around $L_i$. It is well-known that the kernel of $H_1(T_i)\to H_1(S^3\sm L)$ is spanned by the longitude $\l_i$ of $L_i$. 
Since $\pi_1(S^3\sm L)$ is a free group and since every abelian subgroup of a free group is cyclic it now follows easily that the 
longitude also lies in the kernel of $\pi_1(T_i)\to \pi_1(S^3\sm L)$. It now follows from Dehn's lemma that the longitude bounds in fact an embedded disk in $S^3\sm L$. 
\end{proof}

Note that if a finitely presented group is free, then one can show this using Tietze moves. On the other hand there is in general no algorithm for showing that a finitely presented group is not a free group. Our main theorem now gives in particular an algorithm for showing that a given link group is not free.

\begin{theorem}
An $m$-component link $L$ is the trivial link if and only if
for any representation $\a\co \pi_1(S^3\sm L)\to \gl(k,\Q)$
we have $\op{rk}(L,\a)=k(m-1)$ and $\ti{\Delta}_L^\a(t) = 1$.
\end{theorem}

\begin{proof}
The proof of the `only if' statement is very similar to the proof of Theorem \ref{thm:split} (1).
In fact it follows easily from (\ref{equ:mv}) that
for the  $m$-component trivial link $L$ and a  representation $\a\co \pi_1(S^3\sm L)\to \gl(k,\Q)$
we have $H_1(S^3\sm L;\qkt)\cong \qt^{k(m-1)}$. In particular $\op{rk}(L,\a)=k(m-1)$ and $\ti{\Delta}_L^\a(t) = 1$.

We now suppose that $L=L_0\cup \dots \cup L_{m-1}$ is  an $m$-component link  such that
for every representations $\a\co \pi_1(S^3\sm L)\to \gl(k,\Q)$
we have $\op{rk}(L,\a)=k(m-1)$.
It follows immediately from Theorem \ref{thm:split} (2) that $L$ is an $(m-1)$-split link.
We can therefore pick disjoint $3$-balls $B_1,\dots,B_{m-1}\subset S^3$
such that each $B_i$ contains a component of $L$ and such that $B_0:=S^3\sm (B_1\cup \dots \cup B_s)$ also contains a component of $L$.
Without loss of generality we can assume that for $i=0,\dots,m-1$ we have $L_i=L\cap B_i$.
For $i=1,\dots,m-1$ we furthermore write $S_i:=\partial B_i$.

It remains to show that if one of the components $L_i$ is not the unknot,
then there exists a representations $\a\co \pi_1(S^3\sm L)\to \gl(k,\Q)$
with  $\ti{\Delta}_L^\a(t) \neq 1$.
So we now suppose that $L_0$ is not the unknot. It follows from Theorem \ref{thm:detectx} and from Corollary \ref{cor:unknothopf} that for $i=0,\dots,m-1$ there exists a representation
$\a_i:\pi_1(S^3\sm L_i)\to \gl(k_i,\Q)$ such that $\Delta_{S^3\sm L_i}^{\a_i}\ne 0$
and such that $\Delta_{S^3\sm L_0}^{\a_0}$ is not a constant. As in the proof of  Theorem \ref{thm:split}
we can assume that $k:=k_0=\dots=k_{m-1}$.
We then  denote by
\[ \a\co \pi_1(S^3\sm L)\to \gl(k,\Q)\]
the unique representation which has the property that for $i=0,\dots,m-1$ the restriction of $\a$
to $\pi_1(B_i\sm L_i)$ agrees with the restriction of $\a_i$ to $\pi_1(B_i\sm L_i)$.

It now follows from (\ref{equ:mv}) that
\[ \Tor_{\qt}(H_1(S^3\sm L;\qkt))\cong \Tor_{\qt}\left(\bigoplus\limits_{i=0}^{m-1} H_1(S^3\sm L_i;\qkt)\right).\]
We now conclude that
\[\ba{rcl} \ti{\Delta}_L^\a(t)&=&\ord_{\qt}\left(\Tor_{\qt}(H_1(S^3\sm L;\qkt))\right) \\
&=& \ord_{\qt}\left(\Tor_{\qt}\left(\bigoplus\limits_{i=0}^{m-1} H_1(S^3\sm L_i;\qkt)\right)\right) \\
&=& \prod\limits_{i=0}^{m-1} \ord_{\qt}\left(\Tor_{\qt}\left(H_1(S^3\sm L_i;\qkt)\right)\right) \\
&=& \prod\limits_{i=0}^{m-1} \ord_{\qt}\left(H_1(S^3\sm L_i;\qkt)\right) \\
&=& \prod\limits_{i=0}^{m-1} \Delta_{L_i}^\a(t)\\
&=& \prod\limits_{i=0}^{m-1} \Delta_{L_i}^{\a_i}(t).\ea \]
But this is not a constant since $\Delta_{L_0}^{\a_0} (t)$ is not a constant.
\end{proof}

%=====================================
\section{Extending the results}

Let $L$ be an $s$-split. We pick disjoint $3$-balls $B_1,\dots,B_s\subset S^3$
such that each $B_i$ contains a component of $L$ and such that $B_0:=S^3\sm (B_1\cup \dots \cup B_s)$ also contains a component of $L$.
For $i=0,\dots,s$ we write $L_i:=L\cap B_i$.
We then view $L_0,\dots,L_s$ as links in $S^3$. This set of links are called the \emph{split-components} of $L$.
It is well-known that the set of split-components is well-defined and does not depend on the choice of the $B_1,\dots,B_s$.

As a consequence of the proofs of
Corollary \ref{cor:unknothopf}, Theorems \ref{thmc} and \ref{thmatext}, it is rather straightforward to see that twisted Alexander modules determine any $s$-split link such that each of the split-components is either the unknot, the trefoil, the figure-8 knot or the Hopf link.

This result now begs the following question:

\begin{question}
Are there any other links which are determined by twisted Alexander modules?
\end{question}

%=========================================================

\end{document}